\documentclass[reqno]{amsart}
\usepackage{amssymb,latexsym,amsmath,amsthm,enumerate,amsbsy}
\usepackage[mathscr]{eucal}
\usepackage{framed,color,graphicx}
\usepackage{mathrsfs}
\usepackage[all]{xy}
\usepackage{tikz}
\usepackage[sort,compress]{cite}
\usetikzlibrary{positioning,decorations.pathreplacing,patterns,decorations.pathmorphing}
\tikzset{%
element/.style={draw, shape=circle, fill=white, inner sep=1.4pt}
}

\DeclareSymbolFont{bbold}{U}{bbold}{m}{n}
\DeclareSymbolFontAlphabet{\mathbbold}{bbold}

\theoremstyle{plain}
\newtheorem{thm}{Theorem}[section]

\newtheorem{lemma}[thm]{Lemma}
\newtheorem{cor}[thm]{Corollary}
\newtheorem{corollary}[thm]{Corollary}
\newtheorem{pro}[thm]{Proposition}
\newtheorem{proposition}[thm]{Proposition}

\newtheorem{problem}[thm]{Problem}

\theoremstyle{definition}

\renewcommand{\ge}{\geqslant}

\newcommand{\ba}{\mathbf{a}}

\newcommand{\bp}{\mathbf{p}}
\newcommand{\bq}{\mathbf{q}}

\newcommand{\bt}{\mathbf{t}}
\newcommand{\bu}{\mathbf{u}}
\newcommand{\bv}{\mathbf{v}}
\newcommand{\bw}{\mathbf{w}}

\begin{document}

\title[A Continuum in the Lattice of Semiring Varieties]
{A Continuum in the Lattice of Semiring Varieties: The Interval $[\mathsf{V}(S), \mathsf{V}(S^0)]$}

\author{Zidong Gao}
\address{School of Mathematics, Northwest University, Xi'an, 710127, Shaanxi, P.R. China}
\email{zidonggao@yeah.net}

\author{Miaomiao Ren}
\address{School of Mathematics, Northwest University, Xi'an, 710127, Shaanxi, P.R. China}
\email{miaomiaoren@yeah.net}

\author{Xianzhong Zhao}
\address{School of Mathematics, Northwest University, Xi'an, 710127, Shaanxi, P.R. China}
\email{zhaoxz@nwu.edu.cn}


\begin{abstract}
For an additively idempotent semiring (ai-semiring) $S$,
let $S^0$ denote the ai-semiring obtained from $S$ by adjoining a new element $0$.
In this paper,
we develop an approach to investigate the interval $[\mathsf{V}(S), \mathsf{V}(S^0)]$
of ai-semiring varieties between the variety generated by $S$ and that generated by $S^0$.
We establish a general sufficient condition under which this interval has the cardinality of the continuum.
This is applied in particular to $[\mathsf{V}(S_7), \mathsf{V}(S_7^0)]$,
where $S_7$ is a $3$-element ai-semiring and is a nonfinitely based algebra of the smallest possible order,
thereby resolving an open problem proposed by Jackson, Ren, and Zhao (J. Algebra \textbf{611} (2022), 211--245).
The same conclusion holds for $[\mathsf{V}(B_2^1), \mathsf{V}((B_2^1)^0)]$,
where $B_2^1$ is the ai-semiring whose multiplicative reduct is the $6$-element Brandt semigroup.
We also present a sufficient condition for the nonfinite basis property in ai-semiring varieties.
As a corollary, we obtain a new proof of Dolinka's theorem
(Internat. J. Algebra Comput. \textbf{17} (2007), no.~8, 1537--1551) that the $7$-element ai-semiring $(B_2^1)^0$
has no finite basis for its identities.
\end{abstract}

\maketitle

\section{Introduction and preliminaries}\label{sec:intro}

An \emph{additively idempotent semiring} (ai-semiring for short\footnote{The prefix ``ai'' stands for \emph{additively idempotent} and is unrelated to artificial intelligence.})
is an algebra $(S, +, \cdot)$ with two binary operations $+$ and $\cdot$
such that the additive reduct $(S, +)$ is a commutative idempotent semigroup,
the multiplicative reduct $(S, \cdot)$ is a semigroup and the distributive laws
\[
x(y+z)\approx xy+xz,\quad (x+y)z\approx xz+yz
\]
hold.
Such algebras are ubiquitous in mathematics and find applications in diverse areas such as
algebraic geometry~\cite{cc}, tropical geometry~\cite{ms}, information science~\cite{gl}, and theoretical computer science~\cite{go}.

Let $S$ be an ai-semiring. Then the relation $\leq$ on $S$ defined by
\[
a \leq b \Leftrightarrow a+b=b
\]
is a partial order that makes $(S, \leq)$ an upper semilattice, where the supremum of any two elements $a$ and $b$ is $a+b$.
Consequently, the additive reduct $(S, +)$ is uniquely determined by this semilattice order.
Therefore, it is often convenient to represent the addition operation using the Hasse diagram of $(S, \leq)$.
The order $\leq$ is readily seen to be compatible with multiplication,
which is why an ai-semiring is also called a \emph{semilattice-ordered semigroup}.

Let $\mathcal{V}$ be a \emph{variety} of ai-semirings, that is, a class of ai-semirings
that is closed under taking subalgebras, homomorphic images and arbitrary direct products.
By Birkhoff's celebrated theorem, $\mathcal{V}$ is an \emph{equational class},
that is, the class of all ai-semirings that satisfy some set of identities.
The variety $\mathcal{V}$ is \emph{finitely based} if it can be defined by a finite set of identities;
otherwise, it is \emph{nonfinitely based}.
An ai-semiring $S$ is finitely based (or nonfinitely based) if the variety $\mathsf{V}(S)$ it generates possesses this property.

The finite basis problem for a class of ai-semirings concerns the classification of
its members according to whether they are finitely based.
Over the past two decades, this problem has been extensively investigated,
leading to substantial progress. For a detailed overview of results and developments,
see \cite{dol07, dol09, dol092, dol093, dgv, gjrz, gpz05, jrz, pas05, rlzc, rjzl, rlyc, rz16, rzs20, rzw, sr, shap23, vol21, wrz, wzr, yrzs, zrc, zw}.
In particular,
Pastijn et al.~\cite{gpz05, pas05} established that every ai-semiring satisfying the identity $x^2\approx x$ is finitely based.
This result was extended by Ren et al.~\cite{rz16, rzw}, who proved the same for all ai-semirings satisfying
the identity $x^3 \approx x$.
Furthermore, Ren et al.~\cite{rzs20} showed that if $n > 1$ is a square-free integer, then
every ai-semiring satisfying the identities $x^n \approx x$ and $xy\approx yx$ is also finitely based.

The finite basis problem for ai-semirings of small order has attracted considerable attention.
Dolinka~\cite{dol07} constructed the first example of a nonfinitely based finite ai-semiring,
which contains $7$ elements. For the smallest cases,
Shao and Ren~\cite{sr} proved that every ai-semiring in the variety generated by all $2$-element ai-semirings is finitely based. Subsequently, Zhao et al.~\cite{zw} showed that all ai-semirings of order three are finitely based,
with the possible exception of the semiring $S_7$; see Table~\ref{tb24111401} for its Cayley tables.

\begin{table}[ht]
\caption{The Cayley tables of $S_7$} \label{tb24111401}
\begin{tabular}{c|ccc}
$+$      &$0$&$a$&$1$\\
\hline
$0$ &$0$&$0$&$0$\\
$a$      &$0$&$a$&$0$\\
$1$      &$0$&$0$&$1$\\
\end{tabular}\qquad
\begin{tabular}{c|ccc}
$\cdot$  &$0$&$a$&$1$\\
\hline
$0$      &$0$&$0$&$0$\\
$a$      &$0$&$0$&$a$\\
$1$      &$0$&$a$&$1$\\
\end{tabular}
\end{table}

Jackson et al.~\cite{jrz} later settled this remaining case by establishing that $S_7$ itself is nonfinitely based--a corollary of a more general result which also showed that $S_7$ can transmit this property to other finite ai-semirings.
Examples include certain finite flat semirings contained in $\mathsf{V}(S_7)$ and some other finite ai-semirings whose varieties contain $S_7$, such as finite flat semirings, and the ai-semiring $B^1_2$ whose multiplicative reduct is the $6$-element Brandt monoid.
The elements of $B^1_2$ are the matrices
\[
0=\begin{bmatrix}
0 & 0 \\
0 & 0
\end{bmatrix}, \quad
1=\begin{bmatrix}
1 & 0 \\
0 & 1
\end{bmatrix}, \quad
a=\begin{bmatrix}
0 & 1 \\
0 & 0
\end{bmatrix}, \quad
b=\begin{bmatrix}
0 & 0 \\
1 & 0
\end{bmatrix},
\]
\[
ab=
\begin{bmatrix}
1 & 0 \\
0 & 0
\end{bmatrix},\quad
ba=
\begin{bmatrix}
0 & 0 \\
0 & 1
\end{bmatrix},
\]
with multiplication given by matrix multiplication.
Its additive structure is determined by the Hasse diagram in Figure~\ref{fig1}.
One can readily verify that that $B^1_2$ contains a copy of $S_7$.
We note that Volkov~\cite{vol21} independently resolved the finite basis problem for $B^1_2$ using a different method.

\begin{figure}[h]
\centering
\scalebox{0.8}{%
\begin{tikzpicture}[
    node distance=1.5cm and 1.5cm,
    solidnode/.style={circle, draw=black, fill=black, inner sep=2pt, minimum size=4pt},
    hollownode/.style={circle, draw=black, fill=white, inner sep=2pt, minimum size=4pt}
]

\node[solidnode, label=below:1] (1) at (0,0.5) {};

\node[hollownode, label=left:{$a$}] (a) at (-2,2) {};
\node[solidnode, label=below left:{$ab$}] (ab) at (-0.8,2) {};
\node[solidnode, label=below right:{$ba$}] (ba) at (0.8,2) {};
\node[hollownode, label=right:{$b$}] (b) at (2,2) {};

\node[solidnode, label=above:0] (0) at (0,3.5) {};

\draw (1) -- (ab);
\draw (1) -- (ba);

\draw (a) -- (0);
\draw (ab) -- (0);
\draw (ba) -- (0);
\draw (b) -- (0);

\end{tikzpicture}%
}
\caption{The additive order of $B_2^1$}
\label{fig1}
\end{figure}
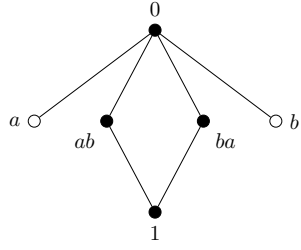

These results motivated Gao et al.~\cite{gjrz} to systematically investigate the finite basis problem for ai-semirings within $\mathsf{V}(S_7)$.
Their work revealed that although $\mathsf{V}(S_7)$ contains $2^{\aleph_0}$ subvarieties,
only $6$ of them are finitely based. Despite this progress, the following fundamental problem remains open:
\begin{problem}\label{problem25112101}
Is every finite ai-semiring whose variety contains $S_7$ nonfinitely based$?$
\end{problem}

Let $S$ be an ai-semiring.
One can construct an ai-semiring $S^0$ from $S$ by adjoining a new element $0$.
The operations on $S^0=S \cup \{0\}$ are defined by:
\[
(\forall a\in S\cup \{0\}) \quad a+0=0+a=a,\quad a0=0a=0,
\]
while preserving the original operations on $S$.
Then $S$ is a subsemiring of $S^0$, and the element
$0$ is both the additive least element and the multiplicative zero of $S^0$.
In particular,
if $S$ is a trivial ai-semiring, then $S^0$ is a $2$-element distributive lattice
and is denoted by $D_2$.
By \cite[Proposition 1.4]{wrz}, $D_2$ is contained in every variety of the form $\mathsf{V}(S^0)$.
The fundamental construction $S^0$ has played a crucial role in the study of
the finite basis problem for ai-semirings,
see, for example,~\cite{gpz05, pas05, rz16, rzs20, rzw}.
Jackson et al.~\cite[Problem 7.4]{jrz} proposed the following problem:
\begin{problem}\label{prob123}
\hspace*{\fill}
\begin{enumerate}[$(1)$]
\item Is $S_7^0$ finitely based or not finitely based$?$
\item What is the cardinality of the interval $[\mathsf{V}(S_7),\mathsf{V}(S_7^0)]$ in the lattice of semiring varieties$?$
\item Under what conditions is the $($non$)$finite basis property preserved when passing from an ai-semiring $S$ to $S^0$$?$
\end{enumerate}
\end{problem}

Motivated by Problem~\ref{prob123}(2), it is natural to consider the following more general problem:
\begin{problem}\label{prob1123}
Let $S$ be an ai-semiring. What is the cardinality of the interval $[\mathsf{V}(S),\mathsf{V}(S^0)]$ in the lattice of semiring varieties$?$
\end{problem}

Let $\mathcal{V}_1$ and $\mathcal{V}_2$ be ai-semiring varieties with $\mathcal{V}_1\subseteq\mathcal{V}_2$.
Then the interval $[\mathcal{V}_1, \mathcal{V}_2]$ denotes the class of all varieties $\mathcal{V}$ such that
$\mathcal{V}_1\subseteq \mathcal{V} \subseteq \mathcal{V}_2$.
Problem~\ref{prob123} (1) constitutes a special case of the more general Problem~\ref{problem25112101}. Wu et al.~\cite{wrz} employed a syntactic approach to prove that every variety in the interval $[\mathsf{V}(S_7, D_2), \mathsf{V}(S_7^0)]$ is nonfinitely based,
where $\mathsf{V}(S_7, D_2)$ denotes the variety generated by $S_7$ and $D_2$.
This result not only resolves Problem~\ref{prob123} (1) by establishing the nonfinite basis property for $S_7^0$ itself but also provides a partial solution to Problem~\ref{problem25112101}.

In a related development, Gao et al.~\cite{gjrz} obtained a sufficient condition for an ai-semiring variety containing $S_7$ to be nonfinitely based.
As applications,
they showed that the entire interval $[\mathsf{V}(S_7), \mathsf{V}(S_7^0)]$
consists of nonfinitely based varieties and established the same property for certain $4$-element ai-semirings whose varieties contain $S_7$.
Gusev and Volkov~\cite{gv2301, gv2302, gv2501} identified mild conditions that guarantee an ai-semiring variety containing $B_2^1$ is nonfinitely based.
The $7$-element ai-semiring $(B_2^1)^0$, denoted by $\Sigma_7$ in Dolinka's paper~\cite{dol07},
was shown to be nonfinitely based.
Moreover, Dolinka~\cite{dol09} proved that every locally finite variety containing
the semiring of $2\times 2$ matrices over $D_2$ is nonfinitely based.
This matrix semiring contains a copy of $(B_2^1)^0$.
Cumulatively, these findings strongly suggest an affirmative answer to Problem~\ref{problem25112101}.

The references~\cite{gpz05, pas05, rz16, rzs20, rzw} contain examples of finite intervals $[\mathsf{V}(S),\mathsf{V}(S^0)]$,
but no examples of such intervals containing infinitely many varieties were previously known.
In Section 2, we develop a general approach to studying such intervals.
Building on this, Section 3 establishes a sufficient condition under which the interval $[\mathsf{V}(S),\mathsf{V}(S^0)]$
has cardinality $2^{\aleph_0}$.
As an application, we prove that $[\mathsf{V}(S), \mathsf{V}(S^0)]$ has this property whenever
$\mathsf{V}(S) \in [\mathsf{V}(S_7), \mathsf{V}(B^1_2)]$. In particular, both $[\mathsf{V}(S_7), \mathsf{V}(S^0_7)]$ and $[\mathsf{V}(B^1_2),\mathsf{V}((B^1_2)^0)]$ satisfy this condition.
These results not only resolve Problem~\ref{prob123} (2) but also advance the study of Problem~\ref{prob1123}.
Furthermore,
we show that there exist $2^{\aleph_0}$ pairwise distinct ai-semirings $S$
such that each interval $[\mathsf{V}(S), \mathsf{V}(S^0)]$ has
cardinality $2^{\aleph_0}$;
moreover, the varieties $\mathsf{V}(S)$ corresponding to these semirings are also pairwise distinct.

In Section 4, we present a sufficient condition for the nonfinite basis property of ai-semiring varieties.
As a consequence, we show that every variety in the interval $[\mathsf{V}(S_7), \mathsf{V}((B^1_2)^0)]$ is nonfinitely based,
and that for certain specific ai-semirings $S$, both $S$ and $S^0$ are nonfinitely based.
These results not only contributes directly to Problem~\ref{problem25112101} and Problem~\ref{prob123}(3),
but also provides a new proof of Dolinka's theorem that $(B_2^1)^0$ is nonfinitely based.

In the remainder of this section, we collect the necessary preliminaries for the subsequent sections.
A \emph{flat semiring} is an ai-semiring whose additive reduct is
a semilattice of height $1$ with the top element being the multiplicative zero.
Equivalently, an ai-semiring is a flat semiring if
its multiplicative reduct has a zero element $0$
and satisfies $a+b=0$ for all distinct elements $a, b\in S$.
Such algebras have been instrumental in advancing the finite basis problem for ai-semirings in recent years,
see~\cite{gjrz, jrz, rjzl, wzr}.
In this context, if $S$ is a flat semiring,
we denote by $\infty$ the top element of $S^0$ to distinguish it from the newly adjoined element.

Jackson et al.~\cite[Lemma 2.2]{jrz} observed that a semigroup $S$ with the zero element $0$
becomes a flat semiring with the top element $0$ if and only if it is
$0$-cancellative, that is, $ab=ac\neq0$ implies $b=c$ and $ba=ca\neq0$ implies $b=c$ for all $a, b, c\in S$.
Consequently, constructing a specific flat semiring reduces to specifying a $0$-cancellative semigroup.

The following algebras form an important class of flat semirings.
Let $W$ be a nonempty subset of a free commutative semigroup,
and let $S_c(W)$ denote the set of all nonempty subwords of words in
$W$ together with a new symbol $0$. Define a binary operation $\cdot$ on $S_c(W)$ by the rule
\[
\bu\cdot \bv=
\begin{cases}
\bu\bv,& \text{if }~\bu\bv\in S_c(W)\setminus \{0\}, \\
0,& \text{otherwise.}
\end{cases}
\]
Then $(S_c(W), \cdot)$ forms a commutative semigroup with zero element $0$.
It is easy to check that $(S_c(W), \cdot)$ is $0$-cancellative and so $S_c(W)$ becomes a flat semiring.
In particular, if~$W$ consists of a single word $\bw$ we shall write $S_c(W)$ as $S_c(\bw)$.
If we allow the empty word in this construction, then the semigroup reduct is a monoid.
The corresponding flat semiring is denoted by $M_c(W)$.
If $a$ is a letter, then $M_c(a)$ is isomorphic to $S_7$.
For any natural number $k\geq 1$, let $a_1\cdots a_k$ denote a linear word of length $k$.
The following proposition summarizes some properties of the flat semiring $S_c(a_1\cdots a_k)$.
\begin{pro}\label{sca1ak}
For any natural number $k \geq 1$, $S_c(a_1\cdots a_k)$ belongs to $\mathsf{V}(S_7)$.
Moreover, $\mathsf{V}(S_c(a_1\cdots a_k))$ is properly contained in $\mathsf{V}(S_c(a_1\cdots a_{k+1}))$.
\end{pro}
\begin{proof}
This statement is immediate from \cite[Figure 1]{gjrz}.
\end{proof}

Let $X^+$ denote the free semigroup over a countably infinite set $X$ of variables.
By distributivity, all ai-semiring terms over $X$ can be expressed as a finite sum of words from words in $X^+$.
An \emph{ai-semiring identity} over $X$ is an
expression of the form
\[
\bu\approx \bv,
\]
where $\bu$ and $\bv$ are ai-semiring terms over $X$.
From \cite[Theorem 2.5]{kp} we know that the ai-semiring
$(P_f(X^+), \cup, \cdot)$ consisting of all non-empty finite subsets of $X^+$
is free in the variety of all ai-semirings.
So we sometimes write
\[
\{\bu_i \mid 1 \leq i \leq k\}\approx \{\bv_j \mid 1 \leq j \leq \ell\}
\]
for the ai-semiring identity
\[
\bu_1+\cdots+\bu_k\approx \bv_1+\cdots+\bv_\ell.
\]
Let $\bu\approx \bv$ be an ai-semiring identity over an alphabet $\{x_1, x_2, \ldots, x_n\}$.
We say that an ai-semiring \emph{$S$ satisfies $\bu\approx \bv$} or \emph{$\bu\approx \bv$ holds in $S$},
if $\varphi(\bu)=\varphi(\bv)$ for all semiring homomorphisms $\varphi: P_f(X^+)\to S$,
that is, $\bu(a_1, a_2, \ldots, a_n)=\bv(a_1, a_2, \ldots, a_n)$ for all $a_1, a_2, \ldots, a_n\in S$.

For an ai-semiring term $\mathbf{u}$, we write $\mathbf{u} \approx 0$ in Section~3
to denote the identities $\mathbf{u} x \approx x \mathbf{u} \approx \mathbf{u}$, where $x$ is a variable not occurring in $\mathbf{u}$.
If a flat semiring $S$ satisfies $\mathbf{u} \approx 0$,
then $S$ also satisfies $\mathbf{u}\approx \mathbf{u}+\mathbf{q}$ for every word $\mathbf{q}$,
since the multiplicative zero of $S$ is also its additive top element.

Let $\mathbf{u} \approx \mathbf{v}$ be an ai-semiring identity such that
\[
\mathbf{u} = \mathbf{u}_1 + \dots + \mathbf{u}_k, \quad \mathbf{v} = \mathbf{v}_1 + \dots + \mathbf{v}_\ell,
\]
where $\mathbf{u}_i, \mathbf{v}_j \in X^+$ for $1 \leq i \leq k$ and $1 \leq j \leq \ell$.
Then the ai-semiring variety defined by $\mathbf{u} \approx \mathbf{v}$
coincides with the variety defined by the simpler identities
\[
\mathbf{u} \approx \mathbf{u} + \mathbf{v}_j, \quad \mathbf{v} \approx \mathbf{v} + \mathbf{u}_i
\]
for all $1 \leq i \leq k$ and $1 \leq j \leq \ell$.
Therefore, it is both necessary and convenient to work with identities of the form
$\mathbf{u} \approx \mathbf{u} + \mathbf{q}$, where $\mathbf{q}$ is a word.

Let $\bp$ be a word in $X^+$, and let $x$ be a letter in $X$. Then
\begin{itemize}
\item $c(\bp)$ denotes the set of all variables that occur in $\bp$;

\item $occ(x, \bp)$ denotes the number of occurrences of $x$ in $\bp$.
\end{itemize}
For an ai-semiring term $\bu=\bu_1+\cdots+\bu_n$ with each $\bu_i \in X^+$,
we shall use $c(\bu)$ to denote the set of all variables that occur in $\bu$, that is,
\[
c(\bu)=\bigcup_{1\leq i \leq n}c(\bu_i).
\]

For other notions and terminology used in this paper, we refer the reader to
Jackson et al.~\cite{jrz} and Ren et al.~\cite{rjzl} for semiring theory,
and to Burris and Sankappanavar~\cite{bs} for universal algebra.
We shall assume that the reader is familiar with the basic results in these areas.

\section{From semiring $S$ to $S^0$}
In this section, we develop an approach to study the interval $[\mathsf{V}(S),\mathsf{V}(S^0)]$.
For a class $\{S_i \mid i\in I\}$ of ai-semirings,
we denote by $\mathsf{V}(\{S_i \mid i\in I\})$ the variety it generates.
If $\mathcal{V}$ is an ai-semiring variety, then $\mathcal{L}(\mathcal{V})$
denotes the lattice of subvarieties of $\mathcal{V}$.

Let $\mathbf{u}=\mathbf{u}_1+\mathbf{u}_2+\cdots+\mathbf{u}_n$ be an ai-semiring term,
where $\mathbf{u}_i \in X^+$ for $1\leq i \leq n$.
For any word $\mathbf{q}$, let $D_{\mathbf{q}}(\mathbf{u})$ denote the set
$\{\mathbf{u}_i \in \mathbf{u} \mid c(\mathbf{u}_i) \subseteq c(\mathbf{q})\}$.
The following result, which appears as \cite[Proposition 1.5]{wrz},
establishes a fundamental connection between the equational theories of $S^0$ and $S$.

\begin{proposition}\label{s0id}
Let $S$ be an ai-semiring, and let $\mathbf{u} \approx \mathbf{u} + \mathbf{q}$ be an ai-semiring identity
such that $\mathbf{u}=\mathbf{u}_1+\cdots+\mathbf{u}_n$, where $\mathbf{u}_i, \mathbf{q}\in X^+$, $1\leq i \leq n$.
Then the identity $\mathbf{u} \approx \mathbf{u} + \mathbf{q}$ holds in $S^0$
if and only if
$D_{\mathbf{q}}(\mathbf{u})$ is nonempty and the identity
$D_{\mathbf{q}}(\mathbf{u}) \approx D_{\mathbf{q}}(\mathbf{u}) + \mathbf{q}$
is satisfied by $S$.
In particular, if $c(\bu)\subseteq c(\bq)$, then
$\mathbf{u} \approx \mathbf{u} + \mathbf{q}$ holds in $S^0$
if and only if it holds in $S$.
\end{proposition}

Proposition~\ref{s0id} yields the following two corollaries.
\begin{corollary}\label{s0t0}
Let $S$ and $T$ be ai-semirings.
If $S\in \mathsf{V}(T)$, then $S^0\in \mathsf{V}(T^0)$.
\end{corollary}

\begin{corollary}\label{a0b0}
Let $S$ be an ai-semiring, and let $\{T_i\mid i\in I\}$ be a class of ai-semirings.
If $\mathsf{V}(S)=\mathsf{V}(\{T_i\mid i\in I\})$, then $\mathsf{V}(S^0)=\mathsf{V}(\{T_i^0\mid i\in I\})$.
\end{corollary}

For any ai-semiring variety $\mathcal{V}$,
we denote by $\mathcal{V}^0$ the variety generated by the class $\{S^0\mid S\in \mathcal{V}\}$.
It is easy to see that $\mathcal{V}$ is a subvariety of $\mathcal{V}^0$.
If $\mathcal{V}=\mathsf{V}(S)$ for some ai-semiring $S$, then by Corollary \ref{a0b0}, $\mathcal{V}^0=\mathsf{V}(S^0)$.

\begin{corollary}\label{ss0uuq}
Let $\mathcal{V}$ be an ai-semiring variety, and
let $\mathbf{u} \approx \mathbf{u} + \mathbf{q}$ be an ai-semiring identity
such that $\mathbf{u}$ is an ai-semiring term and $\mathbf{q}$ is a word.
If $c(\bu)\subseteq c(\bq)$,
then $\mathcal{V}^0$ satisfies $\bu\approx \bu+\bq$
if and only if $\mathcal{V}$ satisfies it.
\end{corollary}
\begin{proof}
This is also a consequence of Proposition~\ref{s0id}.
\end{proof}

Let $S$ be an ai-semiring.
To characterize the interval $[\mathsf{V}(S),\mathsf{V}(S^0)]$,
it is natural to consider the following mapping
\begin{equation}\label{phi}
\varphi: \mathcal{L}(\mathsf{V}(S)) \to [\mathsf{V}(S),\mathsf{V}(S^0)], \quad \mathcal{V} \mapsto \mathcal{V}^0\vee \mathsf{V}(S),
\end{equation}
where $\mathcal{V}^0\vee \mathsf{V}(S)$ denotes the join of $\mathcal{V}^0$ and $\mathsf{V}(S)$,
that is, the smallest variety containing $\mathcal{V}^0$ and $\mathsf{V}(S)$.
In particular, if $\mathcal{V}=\mathsf{V}(T)$ for some ai-semiring $T$,
then $\mathcal{V}^0=\mathsf{V}(T^0)$ and so
\[
\varphi(\mathcal{V})=\mathcal{V}^0\vee \mathsf{V}(S)=\mathsf{V}(T^0)\vee \mathsf{V}(S)=\mathsf{V}(T^0, S).
\]
Corollary~\ref{s0t0} ensures that that $\varphi$ is well-defined.
Furthermore, one can readily verify that $\varphi$ is order-preserving,
and so the smallest variety in the image of $\varphi$ is $\mathsf{V}(S, D_2)$.

Now, for any subset $X$ of $\mathcal{L}(\mathsf{V}(S))$,
if the restriction of $\varphi$ to $X$ is injective,
then the interval $[\mathsf{V}(S),\mathsf{V}(S^0)]$ has cardinality at least $|X|$.
In particular, if $X$ is uncountable, so is $[\mathsf{V}(S),\mathsf{V}(S^0)]$.
This observation provides a method for determining the cardinality of the interval.

Let $\mathcal{V}$ and $\mathcal{W}$ be distinct subvarieties of $\mathsf{V}(S)$.
To show that $\varphi(\mathcal{V})$ and $\varphi(\mathcal{W})$ are also distinct,
it suffices to find an identity that holds in $S$ and is satisfied by one of $\mathcal{V}^0$ or $\mathcal{W}^0$ but not the other.
Consequently, our task reduces to finding sufficiently many subvarieties $\mathcal{V}$ of $\mathsf{V}(S)$
for which the corresponding varieties $\varphi(\mathcal{V})$ can be distinguished by identities that hold in $S$.

We now focus on the $3$-element ai-semiring $S_7$. A precise understanding of its identities is essential.
Let $\mathbf{u}=\mathbf{u}_1+\cdots+\mathbf{u}_n$ be an ai-semiring term, where $\mathbf{u}_i, \mathbf{q}\in X^+$, $1\leq i \leq n$.
Let $\delta(\bu)$ denote the set
\[
\{Z\subseteq c(\bu)\mid (\forall \bu_i\in \bu)\,c(\bu_i)\cap Z=\{x_i\}, occ(x_i,\bu_i)=1\}.
\]
The following proposition, which is \cite[Proposition 5.5]{jrz}, provides a solution of the equational problem of $S_7$:

\begin{proposition}
Let $\mathbf{u} \approx \mathbf{v}$ be an ai-semiring identity.
Then $\mathbf{u} \approx \mathbf{v}$ holds in $S_7$ if and only if $c(\mathbf{u}) = c(\mathbf{v})$ and
$\delta(\mathbf{u}) = \delta(\mathbf{v})$.
\end{proposition}

As a corollary, we have

\begin{corollary}\label{s7uuq}
Let $\mathbf{u}$ be an ai-semiring term, and let $\mathbf{q}$ be word.
If $\delta(\mathbf{u})$ is empty and $c(\mathbf{q}) \subseteq c(\mathbf{u})$,
then $\mathbf{u} \approx \mathbf{u}+\mathbf{q}$ is satisfied by $S_7$.
\end{corollary}

\begin{lemma}\label{sca1ak0}
Let $k \geq 1$ be a natural number. Then the identity
\[
x_1x_2\cdots x_{k+1} + y^2 \approx x_1x_2\cdots x_{k+1} + y^2 + (x_1x_2\cdots x_{k+1})^2
\]
holds in both $S_7$ and $S_c(a_1\cdots a_k)^0$, but fails in $S_c(a_1\cdots a_{k+1})^0$.
\end{lemma}

\begin{proof}
Let $\mathbf{u} = x_1x_2\cdots x_{k+1} + y^2$ and $\mathbf{q} = (x_1x_2\cdots x_{k+1})^2$.
It is easy to see that $c(\mathbf{q}) \subseteq c(\mathbf{u})$ and $\delta(\mathbf{u})$ is empty.
By Corollary~\ref{s7uuq} it follows that $S_7$ satisfies $\mathbf{u} \approx \mathbf{u} + \mathbf{q}$.

Note that $D_{\mathbf{q}}(\mathbf{u}) = x_1x_2\cdots x_{k+1}$.
Since the multiplicative reduct of $S_c(a_1\cdots a_k)$ is $(k+1)$-nilpotent,
we immediately deduce that $S_c(a_1\cdots a_k)$ satisfies the identity $D_{\mathbf{q}}(\mathbf{u}) \approx D_{\mathbf{q}}(\mathbf{u}) + \mathbf{q}$.
Hence, by Proposition~\ref{s0id}, the identity $\mathbf{u} \approx \mathbf{u} + \mathbf{q}$ holds in $S_c(a_1\cdots a_{k})^0$.

Now, consider the substitution
\[
\varphi\colon \{y, x_1, \ldots, x_{k+1}\} \to S_c(a_1a_2\cdots a_{k+1})
\]
defined by $\varphi(x_i) = a_i$ for $1\leq i\leq k+1$ and $\varphi(y) = 0$.
Then $\varphi(\mathbf{u}) = a_1a_2\cdots a_{k+1}$, but $\varphi(\mathbf{q}) = \infty$.
This implies that $\varphi(\mathbf{u}) \neq \varphi(\mathbf{u} + \mathbf{q})$.
Therefore, $S_c(a_1a_2\cdots a_{k+1})^0$ does not satisfy $\mathbf{u} \approx \mathbf{u} + \mathbf{q}$.
\end{proof}

\begin{thm}\label{pro25112310}
The interval \([\mathsf{V}(S_7),\mathsf{V}(S_7^0)]\) is infinite. More precisely, it contains an infinite strictly ascending chain
\[
\mathsf{V}(S_7, S_c(a_1)^0) \subsetneq  \mathsf{V}(S_7, S_c(a_1a_2)^0) \subsetneq  \mathsf{V}(S_7, S_c(a_1a_2a_3)^0)\subsetneq \cdots.
\]
\end{thm}
\begin{proof}
Let $k\geq 1$ be a natural number.
By Proposition~\ref{sca1ak},
$S_c(a_1a_2\cdots a_k)$ lies in $\mathsf{V}(S_7)$
and $\mathsf{V}(S_c(a_1a_2\cdots a_k))$ is properly contained in $\mathsf{V}(S_c(a_1a_2\cdots a_{k+1}))$.
Since the mapping $\varphi$ defined in~\eqref{phi} is order-preserving,
we have
\[
\varphi(\mathsf{V}(S_c(a_1a_2\cdots a_k)))\subseteq\varphi(\mathsf{V}(S_c(a_1a_2\cdots a_{k+1}))),
\]
and so
\[
\mathsf{V}(S_7, S_c(a_1a_2\cdots a_k)^0)\subseteq \mathsf{V}(S_7, S_c(a_1\cdots a_{k+1})^0).
\]
Furthermore, Lemma~\ref{sca1ak0} implies that this inclusion is proper, which establishes the required result.
\end{proof}

\section{Kneser hypergraph semirings and continuum intervals}\label{sec:blockhypergraph}
In this section, we strengthen Theorem~\ref{pro25112310} by showing that
$[\mathsf{V}(S_7),\mathsf{V}(S_7^0)]$ contains $2^{\aleph_0}$ distinct varieties.
We adopt the approach developed in the preceding section, using an important class of flat semirings introduced by Gao et al.~\cite{gjrz},
known as $k$-uniform Kneser hypergraph semirings.
For the reader's convenience, we first provide some basic background material on these semirings.

A \emph{hypergraph} $\mathbb{H}$ is a pair $(V, E)$,
where $E$ is a family of nonempty subsets of a set $V$.
Each element of $V$ is a \emph{vertex} of $\mathbb{H}$,
and each element of $E$ is a \emph{hyperedge} of $\mathbb{H}$.
A subset of $V$ is a \emph{subhyperedge} of $\mathbb{H}$ if it is contained in some hyperedge.
For an integer $k \geq 1$, the hypergraph $\mathbb{H}$ is \emph{$k$-uniform} if $|e| = k$ for every $e \in E$.

A \emph{homomorphism} from a hypergraph $\mathbb{G}=(V(\mathbb{G}), E(\mathbb{G}))$ to a hypergraph
$\mathbb{H}=(V(\mathbb{H}), E(\mathbb{H}))$ is a mapping
$\varphi \colon V(\mathbb{G}) \to V(\mathbb{H})$ such that for every hyperedge $e \in E(\mathbb{G})$,
its image $\varphi(e)$ is a hyperedge in $E(\mathbb{H})$.
We denote by ${\rm Hom}(\mathbb{G}, \mathbb{H})$ the set of all homomorphisms from $\mathbb{G}$ to $\mathbb{H}$.

We now recall some ai-semiring terms associated with $k$-uniform hypergraphs,
which were originally introduced in \cite{aj, gjrz}.
Let $\mathbb H=(V, E)$ be a $k$-uniform hypergraph,
and let $\{x_v \mid v \in V\}$ be a set of variables in one-to-one correspondence with $V$.
We denote by ${\bf t}_{{\mathbb{H}}}$ the ai-semiring term
\[
\sum_{\{v_1,v_2,\ldots,v_k\} \in E} x_{v_1}x_{v_2}\cdots x_{v_k},
\]
and by $\mathbf{q}_{_{{\mathbb H}}}$ the word
\[
\prod_{v \in V} x_v,
\]
where the product may be taken in any fixed order.
Note that each hyperedge of $\mathbb{H}$ gives rise to $k!$ distinct words in ${\bf t}_{{\mathbb{H}}}$.

Let $m$ and $k \geq 3$ be positive integers.
The \emph{$k$-uniform Kneser hypergraph} $\mathbb{H}_{k,m}=\langle V_{k,m}, E_{k,m}\rangle$
is the $k$-uniform hypergraph whose vertices are the $m$-subsets of the set $[km]=\{1,2,\ldots,km\}$,
and whose hyperedges are the $k$ pairwise disjoint such subsets.
Consequently, each hyperedge corresponds to a partition of $[km]$ into $k$ parts of size $m$.

Let $\overline{E}_{k,m}$ denote the set of all unions of one or more vertices from some hyperedge of $\mathbb{H}_{k,m}$.
Since any disjoint union of $m$-element subsets of $[km]$ can be extended to a hyperedge of $\mathbb{H}_{k,m}$,
$\overline{E}_{k,m}$ consists precisely of all disjoint unions of $m$-element subsets of $[km]$.
We define a multiplication $\cdot$ on the set $\overline{E}_{k, m} \cup \{0\}$ by the following rule:
\[
A \cdot B :=
\begin{cases}
A \cup B, & \text{if } A \cap B = \emptyset, \\
0, & \text{otherwise},
\end{cases}
\]
with
\[
0\cdot A=A\cdot 0=0\cdot 0=0
\]
for all $A, B\in \overline{E}_{k, m}$.
It is straightforward to verify that $(\overline{E}_{k, m}\cup \{0\}, \cdot)$ forms a $0$-cancellative semigroup.
This structure therefore yields a flat semiring,
called the \emph{$k$-uniform Kneser hypergraph semiring} defined by $\mathbb{H}_{k,m}$, and denoted by $S_{{\mathbb{H}_{k,m}}}$.
Furthermore, by construction, $S_{\mathbb{H}_{k,m}}$ is isomorphic to the subsemiring of
$S_c(a_1a_2\cdots a_{km})$ generated by the set
\[
\left\{ \prod_{i \in A} a_i \ \middle| \ A \in V_{k,m} \right\}.
\]
An application of Proposition~\ref{sca1ak} therefore yields the following proposition.
\begin{proposition}\label{pro25112801}
The $k$-uniform Kneser hypergraph semiring $S_{\mathbb{H}_{k,m}}$
belongs to the variety $\mathsf{V}(S_7)$ for all positive integers $m$ and $k\geq 3$.
\end{proposition}

\begin{proposition}\label{erzeyi}
Let $m, n$ and $k$ be integers with $k \geq 3$.
Then $S_{{\mathbb{H}_{k,n}}}$ satisfies the identity $\mathbf{t}_{\mathbb{H}_{k, m}}\approx 0$
if and only if ${\rm Hom}(\mathbb{H}_{k,m},\mathbb{H}_{k,n})$ is empty.
\end{proposition}
\begin{proof}
Suppose that $S_{\mathbb{H}_{k,n}}$ does not satisfy the identity $\mathbf{t}_{\mathbb{H}_{k, m}}\approx 0$.
Then there exists a substitution $\varphi \colon \{x_v\mid v\in V_{k,m}\} \to S_{{\mathbb{H}_{k,n}}}$
such that $\varphi(\mathbf{t}_{\mathbb{H}_{k, m}})\neq 0$.
Since each summand of $\mathbf{t}_{\mathbb{H}_{k,m}}$ is a linear word of length $k$ and
the multiplicative of $S_{{\mathbb{H}_{k,n}}}$ is a $(k+1)$-nilpotent semigroup,
it follows that $\varphi(\mathbf{t}_{\mathbb{H}_{k, m}})=[kn]$,
and so $\varphi(x_{v_1}\dots x_{v_k})=[kn]$ for every $\{v_1, v_2, \ldots, v_k\} \in E_{k, m}$.
This implies that $\varphi$ maps the variables to elements corresponding to vertices, and that
$\{\varphi(x_{v_1}),\dots,\varphi(x_{v_k})\}$ is a hyperedge of $\mathbb{H}_{k,n}$ for every hyperedge $\{v_1,\dots,v_k\}$ of $\mathbb{H}_{k,m}$.
Therefore, there is a homomorphism from $\mathbb{H}_{k,m}$ to $\mathbb{H}_{k,n}$.

Conversely, suppose that $\varphi$ is a homomorphism from $\mathbb{H}_{k,m}$ to $\mathbb{H}_{k,n}$.
Consider the substitution
\[
\psi\colon \{x_v \mid v \in V_{k, m}\} \to S_{\mathbb H_{k, n}}, \quad x_v \mapsto \varphi(v).
\]
Then $\psi(\mathbf{t}_{\mathbb{H}_{k, m}})=[kn]\neq 0$.
Thus $S_{{\mathbb{H}_{k,n}}}$ does not satisfy the identity $\mathbf{t}_{\mathbb{H}_{k, m}}\approx 0$.
\end{proof}

\medskip

\noindent \textbf{Remark.}
Propositions~\ref{pro25112801} and \ref{erzeyi} also follow from Lemmas 4.2 and 4.10 of \cite{gjrz}, respectively.

Let $\varphi$ be a mapping from $V_{k,m}$ to $V_{k,n}$. According to the definitions of hypergraph homomorphism and the Kneser hypergraph, $\varphi$ is a homomorphism from $\mathbb{H}_{k,m}$ to $\mathbb{H}_{k,n}$
if and only if $\varphi$ maps every set of $k$ pairwise disjoint $m$-subsets of $[km]$ to $k$ pairwise disjoint subsets in $[kn]$.
This condition is equivalent to requiring that $\varphi$ preserves disjointness, that is, it maps any two disjoint $m$-subsets of $[km]$ to disjoint $m$-subsets of $[kn]$.

Let $\mathbb{G}_{k, m}$ denote the \emph{Kneser graph}, whose vertices are the $m$-subsets of the set $[km]$,
with two $m$-subsets adjacent if they are disjoint.
The above argument shows that ${\rm Hom}(\mathbb{H}_{k,m},\mathbb{H}_{k,n}) = {\rm Hom}(\mathbb{G}_{k,m},\mathbb{G}_{k,n})$.
This result can also be found in \cite[Theorem 1]{fmm}.
On the other hand, \cite[Lemma 7.9.3]{cg} states that
${\rm Hom}(\mathbb{G}_{k,m},\mathbb{G}_{k,n})$ is nonempty if and only if $m$ divides $n$.
We therefore obtain:

\begin{lemma}\label{homgh}
Let $m, n$ and $k$ be integers with $k \geq 3$. Then
${\rm Hom}(\mathbb{H}_{k,m},\mathbb{H}_{k,n})$ is nonempty if and only if $m$ divides $n$.
\end{lemma}

\begin{lemma}\label{s71}
Let $k \geq 3$ be an integer, and let $p$ and $q$ be primes. Then $S_{\mathbb H_{k,q}}$ satisfies the identity
\begin{equation}\label{thinfty}
{\bf t}_{\mathbb{H}_{k,p}} \approx 0
\end{equation}
if and only if $p \neq q$.
\end{lemma}
\begin{proof}
This follows immediately from Proposition~\ref{erzeyi} and Lemma~\ref{homgh}.
\end{proof}

%
%

\begin{lemma}\label{hkq0}
Let $k \geq 3$ be an integer, and let $p$ and $q$ be primes. Then the identity
\begin{equation}\label{thyq}
{\bf t}_{\mathbb{H}_{k,p}} + y^2 \approx {\bf t}_{\mathbb{H}_{k,p}} + y^2 + {\bf q}_{_{\mathbb{H}_{k,p}}}
\end{equation}
holds in $S_7$, and $S_{{\mathbb H_{k,q}}}^0$ satisfies this identity if and only if $p \neq q$.
\end{lemma}

\begin{proof}
For convenience,
let us denote ${\bf t}_{\mathbb{H}_{k,p}} + y^2$ and ${\bf q}_{_{\mathbb{H}_{k,p}}}$ by $\bu$ and $\bq$, respectively.
Observe that
$
\delta(\bu) = \emptyset \ \text{and} \ c(\bq)
\subseteq c(\bu).
$
By Corollary~\ref{s7uuq}, we immediately deduce that the identity~\eqref{thyq} holds in $S_7$.

Now consider the substitution
$\varphi\colon \{x_v \mid v \in V_{k, p}\} \cup \{y\} \to S_{\mathbb H_{k, p}}^0$
defined by $\varphi(y) = 0$ and $\varphi(x_v) = v$ for all $v \in V_{k, p}$.
Then $\varphi(\bu) = [kp] + 0 = [kp]$,
while $\varphi(\bq) = \infty$.
Thus
$\varphi(\bu) \neq \varphi(\bu+\bq)$,
and so $S_{\mathbb H_{k,p}}^0$ does not satisfy identity~\eqref{thyq}.

Let $q$ be a prime different from $p$. Then by Lemma~\ref{s71},
$S_{\mathbb H_{k,q}}$ satisfies the identity \eqref{thinfty}.
Since the multiplicative zero of $S_{\mathbb H_{k,q}}$ is its additive top element,
it follows that $S_{\mathbb H_{k,q}}$ also satisfies the identity
$
{\bf t}_{\mathbb{H}_{k,p}} \approx {\bf t}_{\mathbb{H}_{k,p}} + \bq.
$
Note that $D_{\bq}(\bu) = {\bf t}_{\mathbb{H}_{k,p}}$.
By Proposition~\ref{s0id} we therefore obtain that $S_{\mathbb H_{k,q}}^0$ satisfies the identity~\eqref{thyq}.
\end{proof}

Throughout this section, let $P$ denote the set of all primes.
For each $k \geq 3$, let $\mathcal{V}_k$ denote the ai-semiring variety defined by the following identities:
\[
\mathbf{t}_{\mathbb{H}_{k,p}} + y^2 \approx \mathbf{t}_{\mathbb{H}_{k,p}} + y^2 + \mathbf{q}_{_{\mathbb{H}_{k,p}}}, \ p \in P.
\]
Lemma \ref{hkq0} implies that that $S_7$ is contained in $\mathcal{V}_k$.
We can now state the key conclusion of this section.

\begin{thm}\label{s7ss0}
Let $k\geq 3$ be an integer, and let $S$ be an ai-semiring such that $\mathsf{V}(S) \in [\mathsf{V}(S_7), \mathcal{V}_k]$.
Then the interval $[\mathsf{V}(S),\mathsf{V}(S^0)]$ has cardinality $2^{\aleph_0}$.
Moreover, it contains both a chain and an antichain of cardinality $2^{\aleph_0}$.
\end{thm}

\begin{proof}
Let $\mathcal{P}(P)$ denote the power set of the set $P$, and let $k \geq 3$ be an integer.
Define a mapping $\psi\colon \mathcal{P}(P) \to [\mathsf{V}(S),\mathsf{V}(S^0)]$ by
\[
\psi(Q) = \mathsf{V}(\{S_{\mathbb{H}_{k,q}}^0 \mid q \in Q\}) \vee \mathsf{V}(S) \quad \text{for } Q \in \mathcal{P}(P),
\]
where $\mathsf{V}(\{S_{\mathbb{H}_{k,q}}^0 \mid q \in Q\}) \vee \mathsf{V}(S)$
denotes the join of $\mathsf{V}(\{S_{\mathbb{H}_{k,q}}^0 \mid q \in Q\})$
and $\mathsf{V}(S)$.
Corollary~\ref{s0t0} and Proposition~\ref{pro25112801} ensure that $\psi$ is well-defined.
We show that $\psi$ is a lattice embedding.

Indeed, let $Q$ be a proper subset of $P$, and let $p$ be a prime that is not in $Q$.
By Lemma~\ref{hkq0}, the identity~\eqref{thyq} is satisfied by $\{S_{\mathbb H_{k,q}}^0 \mid q \in Q\}$,
but fails in $S_{\mathbb H_{k,p}}^0$.
This observation allows us to conclude that $\varphi$ is a lattice embedding.

Since $\mathcal{P}(P)$ has the cardinality $2^{\aleph_0}$,
it follows immediately that the interval $[\mathsf{V}(S),\mathsf{V}(S^0)]$ also has the cardinality $2^{\aleph_0}$.
Moreover, by \cite[Chapter I, \S10]{ku2011}, $\mathcal{P}(P)$ contains both a chain and an antichain of size $2^{\aleph_0}$,
and therefore so does $[\mathsf{V}(S),\mathsf{V}(S^0)]$.
\end{proof}

\noindent\textbf{Remark.} Consider the mapping
\[
\Phi \colon \mathcal{P}(P) \to \mathcal{L}(\mathsf{V}(S)), \quad Q \mapsto \mathsf{V}(\{S_{\mathbb{H}_{k,q}} \mid q \in Q\}).
\]
Then the mapping $\psi$ in the proof of Theorem~\ref{s7ss0} coincides with the composition $\varphi \circ \Phi$, where $\varphi$ is defined in \eqref{phi}.
One can show that $\Phi$ is also a lattice embedding, whence $\mathcal{L}(\mathsf{V}(S))$ contains $2^{\aleph_0}$ distinct subvarieties.
Thus, we have essentially shown that the restriction of $\varphi$ to the image of $\Phi$ is injective,
which establishes that the interval $[\mathsf{V}(S), \mathsf{V}(S^0)]$ has cardinality $2^{\aleph_0}$.

\begin{cor}\label{b212n0}
Let $S$ be an ai-semiring such that $\mathsf{V}(S)\in [\mathsf{V}(S_7), \mathsf{V}(B_2^1)]$.
Then the interval $[\mathsf{V}(S),\mathsf{V}(S^0)]$ has cardinality $2^{\aleph_0}$.
\end{cor}

\begin{proof}
By Theorem~\ref{s7ss0}, it suffices to show that $B_2^1$ belongs to $\mathcal{V}_3$.
To this end, we must prove that $B_2^1$ satisfies the identity
\begin{equation}\label{b213p}
{\bf t}_{\mathbb{H}_{3,p}} + y^2 \approx {\bf t}_{\mathbb{H}_{3,p}} + y^2 + \mathbf{q}_{_{\mathbb{H}_{3,p}}}
\end{equation}
for every prime $p$. Let $\varphi$ be an arbitrary substitution from $\{x_v \mid v\in V_{3,p}\}$ to $B_2^1$.
If $\varphi({\bf t}_{\mathbb{H}_{3,p}} + y^2)=0$, then
\[
\varphi({\bf t}_{\mathbb{H}_{3,p}} + y^2)=\varphi({\bf t}_{\mathbb{H}_{3,p}} + y^2+\mathbf{q}_{_{\mathbb{H}_{3,p}}})=0.
\]
Now suppose that $\varphi({\bf t}_{\mathbb{H}_{3,p}} + y^2)\neq 0$.
Since $\varphi(y^2)\in \{ab, ba, 1\}$, it follows that $\varphi({\bf t}_{\mathbb{H}_{3,p}})\in \{ab,ba,1\}$.
We proceed by considering the following three cases:

\textbf{Case 1.} $\varphi({\bf t}_{\mathbb{H}_{3,p}})=1$.
Then $\varphi(x_ux_vx_w)=1$ for all $\{u,v,w\}\in E_{3, p}$, since $1$ is minimal in $(B_2^1, \leq)$.
This implies that $\varphi(x_u)=\varphi(x_v)=\varphi(x_w)=1$ for all $\{u,v,w\}\in E_{3, p}$,
and so $\varphi( \mathbf{q}_{_{\mathbb{H}_{3,p}}})=1$.
Thus
\[
\varphi({\bf t}_{\mathbb{H}_{3,p}} + y^2)=\varphi({\bf t}_{\mathbb{H}_{3,p}} + y^2+\mathbf{q}_{_{\mathbb{H}_{3,p}}})=1+\varphi(y^2).
\]

\textbf{Case 2.} $\varphi({\bf t}_{\mathbb{H}_{3,p}})=ab$.
Then $\varphi(x_ux_vx_w)\in \{ab,1\}$ for all $\{u,v,w\}\in E_{3, p}$ and
$\varphi(x_{u_1}x_{u_2}x_{u_3})=ab$ for some $\{u_1,u_2,u_3\}\in E_{3, p}$.
It follows that
\[
\varphi(x_{u_{i_1}})\varphi(x_{u_{i_2}})\varphi(x_{u_{i_3}})=ab
\]
for each $\{i_1,i_2,i_3\}=\{1,2,3\}$. Furthermore,
$\{\varphi(x_{u_1}),\varphi(x_{u_2}),\varphi(x_{u_3})\}=\{ab\}$ or $\{ab, 1\}$.
This implies that
$\{ab\}\subseteq \{\varphi(x_v)\mid v\in V_{3, p}\}\subseteq \{ab,1\}$,
and so $\varphi(\mathbf{q}_{_{\mathbb{H}_{3,p}}})=ab$. Thus
\[
\varphi({\bf t}_{\mathbb{H}_{3,p}} + y^2)
=\varphi({\bf t}_{\mathbb{H}_{3,p}} + y^2+\mathbf{q}_{_{\mathbb{H}_{3,p}}})=ab+\varphi(y^2).
\]

\textbf{Case 3.} $\varphi({\bf t}_{\mathbb{H}_{3,p}})=ba$.
This is similar to the preceding case.

In all cases, the two sides of~\eqref{b213p} evaluate equally under $\varphi$.
Therefore, $B_2^1$ satisfies identity~\eqref{b213p}.
\end{proof}

As a consequence of Corollary~\ref{b212n0}, we have
\begin{corollary}\label{s7s70}
The intervals $[\mathsf{V}(S_7),\mathsf{V}(S_7^0)]$ and $[\mathsf{V}(B_2^1),\mathsf{V}((B_2^1)^0)]$
both have cardinality $2^{\aleph_0}$.
\end{corollary}

\begin{proposition}\label{cacsemiring}
There exist $2^{\aleph_0}$ pairwise distinct ai-semirings $S$ such that the intervals
$[\mathsf{V}(S), \mathsf{V}(S^0)]$ have cardinality $2^{\aleph_0}$,
and the corresponding varieties $\mathsf{V}(S)$ are pairwise distinct.
\end{proposition}
\begin{proof}
For each subset $Q$ of $P$, let $S_{Q}$ denote the direct product of
the semirings in $\{S_7\}\cup \{S_{\mathbb H_{3,q}}^0 \mid q \in Q\}$.
By the proof of Theorem~\ref{s7s70},
the collection $\{S_Q \mid Q \subseteq P\}$ forms a family of $2^{\aleph_0}$ semirings that generate pairwise distinct varieties.

To complete the proof, by Theorem~\ref{s7ss0},
it suffices to show that $\mathsf{V}(S_{Q})$ lies in the interval $[\mathsf{V}(S_7), \mathcal{V}_4]$.
Indeed, it is evident that $S_7$ is contained in $\mathsf{V}(S_{Q})$.
It remains to show that $S_{Q} \in \mathcal{V}_4$, i.e., $S_Q$ satisfies the identities
\begin{equation}\label{eq4p}
\mathbf{t}_{\mathbb{H}_{4,p}} + y^2 \approx \mathbf{t}_{\mathbb{H}_{4,p}} + y^2 + \mathbf{q}_{_{\mathbb{H}_{4,p}}}, \ p \in P.
\end{equation}
By Lemma~\ref{hkq0}, $S_7$ satisfies identities~\eqref{eq4p}.
Now let $q$ be an arbitrary prime.
Observe that $S_{\mathbb H_{3,q}}$ satisfies the identity
$\mathbf{t}_{\mathbb{H}_{4,p}} \approx \mathbf{t}_{\mathbb{H}_{4,p}} +\mathbf{q}_{_{\mathbb{H}_{4,p}}}$,
since the multiplicative reduct of $S_{\mathbb H_{3,q}}$ is a $4$-nilpotent semigroup,
and every word in $\mathbf{t}_{\mathbb{H}_{4,p}}$ has length $4$.
Note that $y$ does not occur in $\mathbf{t}_{\mathbb{H}_{4,p}}$.
By Proposition~\ref{s0id}, $S_{\mathbb H_{3,q}}^0$ also satisfies identities~\eqref{eq4p}.
Thus $S_{Q}$ satisfies identities~\eqref{eq4p} and so $S_{Q}$ belongs to $\mathcal{V}_4$.
This proves the required result.
\end{proof}

\section{A sufficient condition for the nonfinitely based property}\label{sec:NFB}
In this section, we present a sufficient condition for an ai-semiring variety to be nonfinitely based.
Analogous to the Kneser hypergraph semirings in Section~3,
we first define a hypergraph semiring $S_{\mathbb{H}}$ for any $k$-uniform hypergraph $\mathbb{H}$
with sufficiently large girth,
whose multiplication yields a commutative $0$-cancellative semigroup that is $(k+1)$-nilpotent.
These semirings were introduced by Jackson et al.~\cite{jrz}
and have proven to be a very useful tool in the study of the finite basis problem (see \cite{aj, gjrz, gr}).
The core idea of the construction is to use the vertices of $\mathbb{H}$ as generators and define multiplication so that it reflects the hyperedge structure:
for every hyperedge, the product of all its vertices yields a fixed common element,
while the product of
any set of vertices that is not a subhyperedge is set to zero.
We begin by reviewing the essential background on these semirings.

A \emph{cycle} in a hypergraph is an alternating sequence
$v_0, e_1, v_1, e_2, v_2, \dots, e_n, v_n$
with $n \ge 2$, $v_0 = v_n$, where $v_0, v_1, \dots, v_{n-1}$ are distinct vertices,
$e_1, \dots, e_n$ are distinct hyperedges, and $\{v_{i-1}, v_i\} \subseteq e_i$ for all $1\leq i \leq n$.
The \emph{length} of this cycle is $n$.
The \emph{girth} of a hypergraph is the length of its shortest cycle,
or infinite if no cycle exists.

Let $\mathbb{H} = (V, E)$ be a $k$-uniform hypergraph with $k \geq 3$,
having no isolated vertices and girth greater than $4$.
The girth condition guarantees that any two distinct hyperedges of $\mathbb{H}$
share at most one vertex, and that a set of vertices of $\mathbb{H}$ is a subhyperedge
if and only if all its $2$-element subsets are subhyperedges (see~\cite[Lemma 3.2]{jrz}).
The \emph{hypergraph semiring} $S_{\mathbb{H}}$ defined by $\mathbb{H}$
is a flat semiring generated by a set $\{\mathbf{a}_v \mid v \in V\}$ in one-to-one correspondence with $V$,
together with a special element~$0$, subject to the following relations:
\begin{itemize}
\item[$(1)$] $0$ is the multiplicative zero.

\item[$(2)$] $\mathbf{a}_u\mathbf{a}_v=\mathbf{a}_v\mathbf{a}_u$ for all $u, v\in V$.

\item[$(3)$] $\mathbf{a}_u\mathbf{a}_v=0$ if $u=v$ or $\{u, v\}$ is not a subhyperedge.

\item[$(4)$] $\mathbf{a}_{u_1}\mathbf{a}_{u_2}\cdots \mathbf{a}_{u_k}=\mathbf{a}_{v_1}\mathbf{a}_{v_2}\cdots \mathbf{a}_{v_k}$
for all $\{u_1, \dots, u_k\}, \{v_1, \dots, v_k\} \in E$; we denote this common value by $\mathbf{a}$.

\item[$(5)$] $\mathbf{a}_{u_1}\mathbf{a}_{u_2}\cdots \ba_{u_{k-1}}
=\mathbf{a}_{v_1}\mathbf{a}_{v_2}\cdots \ba_{v_{k-1}}$ whenever there exists $w \in V$ such that
both $\{u_1, \dots, u_{k-1}, w\}$ and $\{v_1, \dots, v_{k-1}, w\}$ belong to $E$.
\end{itemize}

Let $m\geq 2$ be an integer.
A hypergraph  $\mathbb{H}=(V, E)$ is \emph{$m$-colourable} if there exists a mapping $\varphi: V \to \{1,2,\ldots,m\}$
such that $|\varphi(e)|\geq 2$ for all $e\in E$ with $|e|\geq 2$.
For all integers $k, m, n \geq 3$, we fix a $k$-uniform hypergraph
$\mathbb{H}_{n}^{(k, m)}=(V_n^{(k, m)}, E_n^{(k, m)})$
that has girth greater than $k\binom{kn}{2}$ and is not $m$-colourable.
The existence of such a hypergraph is guaranteed by a result of Erd\"{o}s and Hajnal~\cite{erdhaj}.
For convenience, we denote $\mathbb{H}_{n}^{(k, m)}$ simply by $\mathbb{H}_n =(V_n, E_n)$
when $k$ and $m$ are clear from the context.

Let $k, m\geq 3$ be an integer, and let $\mathcal{W}_{k,m}$ denote the ai-semiring variety defined by the identities
\begin{equation}
\bt_{\mathbb{H}_n}  \approx \bt_{\mathbb{H}_n}+{\bf q}_{_{\mathbb H_n}},\quad n \geq 3.\label{id112500}
\end{equation}
From the proof of \cite[Corollary 2.5]{gjrz}, we have that $\mathcal{W}_{k,m}$ contains $S_7^0$,
and consequently $S_c(a_1a_2\cdots a_k)\in \mathcal{W}_{k,m}$.
The following theorem generalizes \cite[Theorem 2.2]{gjrz}.
While the proof strategy broadly follows that of the cited theorem,
we provide a comprehensive account below to ensure self-containment and enhance readability.

\begin{thm}\label{thm251117}
Let $k, m\geq 3$ be an integer.
Then every variety in the interval $[\mathsf{V}(S_c(a_1\cdots a_k)), \mathcal{W}_{k,m}]$ is nonfinitely based.
\end{thm}

\begin{proof}
Let $\mathcal{V}$ be an arbitrary variety in the interval $[\mathsf{V}(S_c(a_1\cdots a_k)), \mathcal{W}_{k,m}]$.
The proof proceeds as follows.
We show that for every $n \geq 3$, the set of all $n$-variable identities of $\mathcal{V}$
does not form a basis for the equational theory of $\mathcal{V}$.
To establish this,
it suffices to prove that for every $n \geq 3$, the hypergraph semiring $S_{\mathbb{H}_n}$ does not lie in $\mathcal{V}$,
while every $n$-generated subsemiring of $S_{\mathbb{H}_n}$ does lie in $\mathcal{V}$.

Let $n \geq 3$.
From the proof of \cite[Theorem 4.9]{jrz},
we know that every $n$-generated subalgebra $T$ of $S_{\mathbb{H}_n}$
lies in the variety $\mathsf{V}(S_c(a_1 \cdots a_k))$.
Since $\mathsf{V}(S_c(a_1 \cdots a_k))$ is a subvariety of $\mathcal{V}$,
it follows that $T\in \mathcal{V}$.

Now consider the natural substitution
$\varphi\colon \{x_v\mid v\in V_n\} \to S_{\mathbb{H}_n}$ defined by $\varphi(x_v) = \mathbf{a}_v$ for all $v \in V_n$.
It is easy to see that $\varphi(\mathbf{t}_{\mathbb{H}_n}) = \mathbf{a} $ and $\varphi(\mathbf{q}_{_{\mathbb H_n}})=0$.
This shows that $S_{\mathbb{H}_n} $ does not satisfy the identity~\eqref{id112500}.
Since the identity~\eqref{id112500} holds in $\mathcal{W}_{k,m}$
and $\mathcal{V}$ is a subvariety of $\mathcal{W}_{k,m}$,
it follows that $\mathcal{V}$ satisfies the identity~\eqref{id112500}.
Consequently, $S_{_{\mathbb{H}_n}} \notin \mathcal{V}$.


We have thus shown that for every $n \geq 2 $, all $n $-generated subalgebras of $S_{\mathbb{H}_n} $ belong to $\mathcal{V} $, while $ S_{\mathbb{H}_n} $ itself does not.
Therefore, $\mathcal{V} $ is nonfinitely based.
\end{proof}


\begin{cor}\label{coro25120101}
Every variety in the interval $[\mathsf{V}(S_c(abc)),\mathsf{V}((B_2^1)^0)]$ is nonfinitely based.
\end{cor}
\begin{proof}
Since $S_c(abc) \in \mathsf{V}(S_7)$ and $S_7\in \mathsf{V}(B_2^1)$, we obtain that $S_c(abc) \in \mathsf{V}(B_2^1)$,
and so $S_c(abc) \in \mathsf{V}((B_2^1)^0)$.
By Theorem~\ref{thm251117},
it suffices to show that $(B_2^1)^0$ satisfies the identities~\eqref{id112500} for each $n \geq 3$,
where $k=m=3$.
Since $c(\bt_{\mathbb{H}_n})=c({\bf q}_{_{\mathbb H_n}})$,
Proposition~\ref{s0id} implies that we need only prove that the identities~\eqref{id112500} holds in $B_2^1$.

For any $n \geq 3$, let $\varphi\colon \{x_v \mid v \in V_n\} \to B_2^1$ be an arbitrary substitution.
We first show that $\varphi(\mathbf{t}_{\mathbb{H}_n}) \neq a$ and $\varphi(\mathbf{t}_{\mathbb{H}_n}) \neq b$.
Suppose, for contradiction, that $\varphi(\mathbf{t}_{_{\mathbb{H}_n}}) = a$.
Since $a$ is minimal in $(B_2^1, \leq)$,
it follows that $\varphi(x_{u_1}x_{u_2}x_{u_3}) = a$ for every hyperedge $\{u_1, u_2, u_3\} \in E_n$.
This implies that $\varphi(x_{u_{i_1}})\varphi(x_{u_{i_2}})\varphi(x_{u_{i_3}}) = a$
for all $\{i_1,i_2,i_3\} = \{1,2,3\}$.
So $\{\varphi(x_{u_1}), \varphi(x_{u_2}), \varphi(x_{u_3})\}=\{1, 1, a\}$.
Consequently, $\mathbb{H}_n$ is $2$-colourable, which contradicts the choice of $\mathbb{H}_n$.
Therefore, $\varphi(\mathbf{t}_{\mathbb{H}_n}) \neq a$.
A similar argument shows that $\varphi(\mathbf{t}_{\mathbb{H}_n}) \neq b$.

Now assume that $\varphi(\mathbf{t}_{\mathbb{H}_n}) \in \{0, 1, ab, ba\}$.
In this case, the proof of Corollary~\ref{b212n0} shows that
$\varphi(\mathbf{t}_{\mathbb{H}_n}) = \varphi(\mathbf{t}_{\mathbb{H}_n} + \mathbf{q}_{_{\mathbb{H}_n}})$.
Consequently, $B_2^1$ satisfies identity~\eqref{id112500} for every $n \geq 3$.
\end{proof}

\noindent\textbf{Remark}
Corollary~\ref{coro25120101} extends Corollary~\cite[Corollary 2.5]{gjrz},
which asserts that every variety in the interval $[\mathsf{V}(S_c(abc)),\mathsf{V}(S_7^0)]$ is nonfinitely based.

\begin{cor}
Every variety in the interval $[\mathsf{V}(S_7), \mathsf{V}((B_2^1)^0)]$ is nonfinitely based.
\end{cor}
\begin{proof}
Since $\mathsf{V}(S_7) \in [\mathsf{V}(S_c(abc)),\mathsf{V}((B_2^1)^0)]$, the conclusion follows immediately from Corollary~\ref{coro25120101}.
\end{proof}

\begin{cor}
The $7$-element ai-semiring $(B_2^1)^0$ is nonfinitely based.
\end{cor}
\begin{proof}
This is a direct consequence of Corollary~\ref{coro25120101}.
\end{proof}

\begin{corollary}\label{coro25120450}
Let $k, m\geq 3$ be an integer.
If $S$ is an ai-semiring such that $\mathsf{V}(S)\in [\mathsf{V}(S_c(a_1\cdots a_k)), \mathcal{W}_{k, m}]$,
then $S$ and $S^0$ are both nonfinitely based.
\end{corollary}
\begin{proof}\label{coro25113001}
By assumption, $S$ satisfies the identities~\eqref{id112500}.
Since $c(\bt_{\mathbb{H}_n})=c({\bf q}_{_{\mathbb H_n}})$,
Proposition~\ref{s0id} implies that $S^0$ also satisfies identities~\eqref{id112500}.
Hence, $\mathsf{V}(S^0)$ lies in the interval $[\mathsf{V}(S_c(a_1\cdots a_k)), \mathcal{W}_{k, m}]$.
Applying Theorem~\ref{thm251117}, we conclude that $S$ and $S^0$ are both nonfinitely based.
\end{proof}

\begin{corollary}\label{coro251202}
Let $k \geq 3$ be an integer, and let $\mathbb H$ be a $k$-uniform hypergraph.
Then $S_{\mathbb H}$ and $S_{\mathbb H}^0$ are both nonfinitely based.
\end{corollary}
\begin{proof}
Let $\mathbb{H}=(V(\mathbb{H}), E(\mathbb{H}))$.
It is easy to see that
$S_c(a_1\cdots a_k)$ embeds into $S_{\mathbb H}$,
and so $\mathsf{V}(S_c(a_1\cdots a_k))$ is a subvariety of $\mathsf{V}(S_{\mathbb H})$.
By Corollary~\ref{coro25120450}, it is enough to prove that $S_{\mathbb H}\in \mathcal{W}_{k, m}$,
where $m=|V(\mathbb{H})|$.
So we need only to show that $S_{\mathbb H}$ satisfies the identities~\eqref{id112500}.

Fix $n \geq 3$,
and let $\varphi \colon \{x_v \mid v \in V_n\} \to S_{\mathbb{H}}$ be an arbitrary substitution.
If $\varphi(\mathbf{t}_{\mathbb{H}_n}) = \infty$,
then $\varphi(\mathbf{t}_{\mathbb{H}_n} + \mathbf{q}_{_{\mathbb{H}_n}}) = \infty$ because $\infty$ is the additive top element.

Now suppose $\varphi(\mathbf{t}_{\mathbb{H}_n}) \neq \infty$. Then for each word $x_{v_1}\cdots x_{v_k}$ in $\mathbf{t}_{\mathbb{H}_{n}}$, we have that $\varphi(x_{v_1}\cdots x_{v_k}) = \mathbf{a}$.
Hence, for every hyperedge $\{v_1, \dots, v_k\} \in E_n$, there exists $\{u_1,\ldots, u_k\} \in E(\mathbb{H})$ such that
\[
\{\varphi(x_{v_1}), \dots, \varphi(x_{v_k})\} = \{\mathbf{a}_{u_1}, \ldots, \mathbf{a}_{u_k}\}.
\]
It follows that $\mathbb{H}_{n}$ is $m$-colourable, contradicting the choice of $\mathbb{H}_n$.
Thus, the case $\varphi(\mathbf{t}_{\mathbb{H}_n}) \neq \infty$ is impossible, and we must have that
$\varphi(\mathbf{t}_{\mathbb{H}_n)} = \infty$.

Therefore, $S_{\mathbb H}$ satisfies the identities~\eqref{id112500}.
\end{proof}

\noindent\textbf{Remark.} A more general result can be obtained as follows.
Let $k, m \geq 3$ be integers, and let $\mathcal{V}$ be the ai-semiring variety generated by a class
$\{S_{\mathbb{H}_{\lambda}} \mid \lambda \in \Lambda\}$ of $k$-uniform hypergraph semirings,
where each $\mathbb{H}_{\lambda}$ has at most $m$ vertices.
Then both $\mathcal{V}$ and $\mathcal{V}^0$ are nonfinitely based.
The proof follows from Corollary~\ref{ss0uuq} together with the argument used in the proof of Corollary~\ref{coro251202}.




\section{Conclusion}
We establish a general sufficient condition for the interval $[\mathsf{V}(S),\mathsf{V}(S^0)]$ to have cardinality $2^{\aleph_0}$
and apply it to prove that $[\mathsf{V}(S_7),\mathsf{V}(S_7^0)]$ indeed has this property.
This resolves Problem~\ref{prob123} (2) and addresses specific aspects of Problem~\ref{prob1123}.
We also prove that every variety in $[\mathsf{V}(S_7), \mathsf{V}((B_2^1)^0)]$ is nonfinitely based,
which advances our understanding of Problem~\ref{problem25112101}.
Moreover, we show that for certain specific ai-semirings $S$,
both $S$ and $S^0$ are nonfinitely based, thereby making progress on Problem~\ref{prob123} (3).

The references~\cite{gpz05, pas05, rz16, rzs20, rzw} provide examples of ai-semirings $S$
for which both $S$ and $S^0$ are finitely based.
The papers \cite{rjzl, wzr, zw} completely
resolve the finite basis problem for finite ai-semirings of the forms $S_c(W)$ and $S_c(W)^0$.
In particular, $S_c(ab)$ is finitely based, whereas $S_c(ab)^0$ is nonfinitely based.
Thus, the finite basis property for $S$ and $S^0$ does not always coincide.
Nevertheless, we have not yet found an example where $S$ is nonfinitely based but $S^0$ is finitely based.

We expect further progress on Problem~\ref{problem25112101} in the near future,
since it is intimately connected with another ongoing investigation concerning
the finite basis problem for ai-semirings of order four.
Up to isomorphism, there are precisely $866$ such algebras,
which can be categorized into five distinct classes based on their additive orders.
The finite basis problem has been solved for three of these classes (see, e.g., \cite{rlzc, rlyc, yrzs}).
A systematic verification shows that, among these $866$ algebras,
exactly $43$ ai-semirings generate varieties containing $S_7$.
For three of these $43$ ai-semirings, the finite basis problem is still open.
A positive answer to Problem~\ref{problem25112101} would imply that all three are nonfinitely based.

\subsection*{Acknowledgment}
The authors thank Professor Marcel Jackson for his valuable comments on this work.
The authors also thank Jun Jiao, Simin Lyu, Chenyu Yang, Ting Yu, and Mengya Yue for their helpful discussions that contributed to this work.
Miaomiao Ren, corresponding author, is supported by National Natural Science Foundation of China (12371024, 12571020).
Xianzhong Zhao is supported by National Natural Science Foundation of China (12571020).

\bibliographystyle{amsplain}

\end{document}